\documentclass[graybox]{svmult}

\usepackage{mathptmx}       \usepackage{helvet}         \usepackage{courier}        \usepackage{cite}
\usepackage{graphicx}        \usepackage{multicol}        \usepackage[bottom]{footmisc}

\usepackage{fancyvrb}        \usepackage{amsmath}
\usepackage{amssymb}
\usepackage{amsfonts}
\usepackage{newtxtext}
\usepackage{newtxmath}

\usepackage{tikz}
\usepackage{cleveref}

\def\dx{\partial_x}
\def\dy{\partial_y}

\def\m{m}
\def\b{b}
\def\h{h}
\def\n{n}
\def\div{\operatorname{div}}
\def\eps{\epsilon}

\def\RR{\mathbb{R}}

\DeclareSymbolFont{CMletters}{OML}{cmm}{m}{it}
\DeclareMathSymbol{\nu}{\mathord}{CMletters}{23}
\DeclareMathSymbol{v}{\mathord}{CMletters}{`v}
\DeclareMathSymbol{u}{\mathord}{CMletters}{`u}

\begin{document}
\title*{On the stability of harmonic mortar methods with application to electric machines}
\titlerunning{Inf-sup stability of harmonic mortar methods}
\author{H. Egger\and
M. Harutyunyan \and
M. Merkel  \and
S. Sch\"ops }
\institute{Herbert Egger \at Numerical analysis and Scientific Computing, Technische Universit\"at Darmstadt, Dolivostr. 15, 64293 Darmstadt, Germany, 
\email{herbert.egger@tu-darmstadt.de}
  \and M. Harutyunan, M. Merkel, S. Sch\"ops \at Computational Electromagnetis, Technische Universit\"at Darmstadt, Schlossgartenstr. 8,
64289 Darmstadt, Germany, 
  \email{mane.harutyunyan@tu-darmstadt.de, melina.merkel@tu-darmstadt.de, sebastian.schoeps@tu-darmstadt.de}
}

\maketitle

\abstract*{Harmonic stator-rotor coupling offers a promising approach for the interconnection of rotating subsystems in the simulation of electric machines. This paper studies the stability of discretization schemes based on harmonic coupling in the framework of mortar methods. A general criterion is derived that allows to ensure the relevant inf-sup stability condition for a variety of specific discretization approaches, including finite-element methods and isogeometric analysis. The validity and sharpness of the theoretical results is demonstrated by numerical tests.}

\abstract{Harmonic stator-rotor coupling offers a promising approach for the interconnection of rotating subsystems in the simulation of electric machines. This paper studies the stability of discretization schemes based on harmonic coupling in the framework of mortar methods for Poisson-like problems. A general criterion is derived that allows to ensure the relevant inf-sup stability condition for a variety of specific discretization approaches, including finite-element methods and isogeometric analysis with harmonic mortar coupling. The validity and sharpness of the theoretical results is demonstrated by numerical tests.}

\section{Introduction} \label{harutyunyan:sec:1}

Electric drives naturally consist of different subdomains, i.e. the stator and rotor, which move relative to each other. 
The time-varying geometry and nonlinearities caused by saturation effects formally require a time-domain analysis, which is often realized by solving a sequence of quasi-stationary problems at different working points. 
Several strategies have been proposed for the simulation of the corresponding equations of magnetostatics and, in particular, for the coupling of the fields across the air gap between stator and rotor.
As it is common practice, see e.g. \cite{DeGersem04,Gyselinck03,Lange10}, we consider a two dimensional regime, in which the unknown fields are described by the axial component of the magnetic vector potential.  
The governing system then consists of two Poisson-like problems for the stator and the rotor, which can be coupled via Lagrange multipliers. 
Such domain decompositions of mortar methods, which couple subdomains via Lagrange multipliers, have been investigated intensively in the literature \cite{BenBelgacem99,Bernardi94,Braess99,Wohlmuth00}; see \cite{Buffa01,DeGersem04} for results concerning electric machines. 
It is well-known that a careful choice of approximation spaces is required to obtain stable disretization schemes for underlying saddlepoint problems \cite{Brezzi74,Raviart77}; appropriate stabilization \cite{Hansbo05} could be used as an alternative approach. 

In this paper, we investigate the stability of mortar discretizations using harmonic functions for the stator-rotor coupling \cite{DeGersem04,Bontinck18}. We discuss in detail the discrete inf-sup condition which is necessary and sufficient to guarantee the stability of such approximations. 
We provide a simple criterion for the maximal number of harmonics used as Lagrange multipliers depending on the mesh size and polynomial degree of the subdomain discretizations which guarantees the stability of the scheme. Our analysis applies to the harmonic coupling of various discretization methods, e.g., obtained by isogeometric analysis (IGA) \cite{Hughes_2005aa,Bontinck18,Brivadis15}, and can in principle be extended to other Lagrange multiplier spaces.

The remainder of this note is organized as follows: 
In Section~\ref{harutyunyan:sec:2}, we introduce the model problem to be considered and we summarize some well-known results about its analysis and discretization. 
In Section~\ref{harutyunyan:sec:3}, we then turn to the harmonic stator-rotor coupling, and we state and prove our main results. 
Section~\ref{harutyunyan:sec:4} is concerned with numerical tests, in which we demonstrate the validity of our stability criterion for low and high order discretizations based on IGA.

\vspace*{-1em}

\section{Model problem} \label{harutyunyan:sec:2} 

We consider a typical geometric setup that consists of two subdomains $\Omega_1$, $\Omega_2$ representing, respectively, the stator and rotor, separated by a small air gap which contains the interface $\Gamma = \partial\Omega_1 \cap \partial\Omega_2$; see Figure~\ref{harutyunyan:fig:geometry}. 
Let $\Sigma_\ell = \partial\Omega_\ell \setminus \Gamma$, $\ell=1,2$, be the remaining parts of the subdomain boundaries and $f_\ell = f|_{\Omega_\ell}$ denote the restriction of a function $f$ defined on $\Omega_1 \cup \Omega_2$ to the subdomain $\Omega_\ell$. 
We then consider the following elliptic interface problem:
Inside the two subdomains, we require 
\begin{alignat}{5}
    -\div (\nu_\ell  \nabla u_\ell) &= j_\ell, \qquad && \text{in } \Omega_\ell, \label{harutyunyan:eq:sys1}\\
    u_\ell &= 0,    \qquad && \text{on } \Sigma_\ell, \label{harutyunyan:eq:sys2}
\end{alignat}
where $u$ denotes the $z$-component of the magnetic vector potential, $\nu$ the magnetic reluctivity, and $j=j_s+\div \m^\perp$ a generalized current density with $j_s$ denoting the $z$-component of the source currents and $\m^\perp=(\m_y,-\m_x)$ the rotated magnetization vector. 
The corresponding in-plane components of the magnetic flux density and field strength are given by $\b=(\dy u,-\dx u)=\nabla^\perp u$ and $\h=\nu \b$, respectively. 
The coupling of the fields across the interface $\Gamma$ is accomplished by the conditions
\begin{alignat}{5}
    u_1 &= u_2, \qquad && \text{on } \Gamma, \label{harutyunyan:eq:sys3}\\
    \n \cdot (\nu_1 \nabla u_1) &= \n \cdot (\nu_2 \nabla u_2), \qquad && \text{on } \Gamma,  \label{harutyunyan:eq:sys4}
\end{alignat}
which express the normal continuity of $\b$ and the tangential continuity of $\h$, respectively.
\begin{figure}[t]
    \hspace*{-0.5em}
    \centering
    \includegraphics[width=0.45\textwidth]{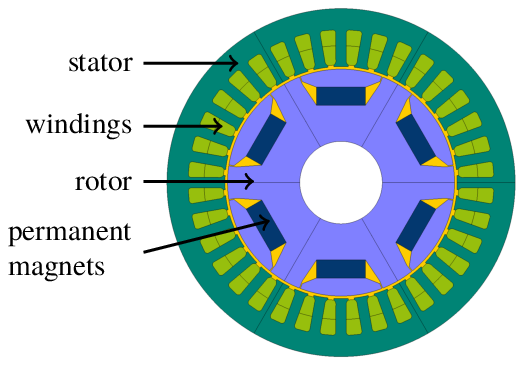}
\hspace*{3.0em}
\includegraphics[width=0.31\textwidth]{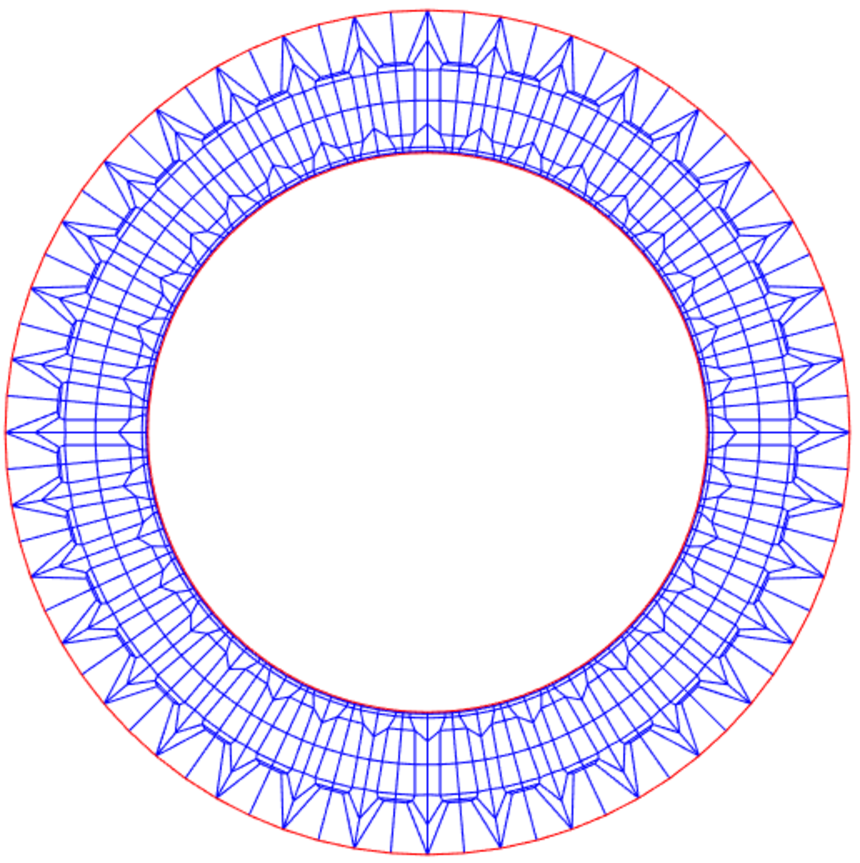}
\caption{Typical structure of a 6-pole permanent magnet synchronous machine (left) and the coarsest mesh of the stator domain as used in our numerical tests (right).\label{harutyunyan:fig:geometry}}
\end{figure}
Here $\n=\n_2$ is the unit normal vector at $\Gamma$ pointing from $\Omega_2$ to $\Omega_1$.
Without further noting, we assume that $\Omega_\ell$ are bounded domains with smooth boundaries $\Sigma_\ell$ and $\Gamma$, having non-zero measure, and that $\nu$ is bounded from above and below by positive constants  $\underline \nu$, $\overline \nu$, i.e.,  $ \underline \nu \le \nu(x) \le \overline \nu$ for all $x \in \Omega_1 \cup \Omega_2$. 

The weak formulation of the interface problem \eqref{harutyunyan:eq:sys1}--\eqref{harutyunyan:eq:sys4} then reads as follows: 
Find $u \in V=\{v \in H^1(\Omega_1 \cup \Omega_2) : v|_{\Sigma_\ell}=0\}$ and $\lambda \in M=H^{-1/2}(\Gamma)$ such that 
\begin{alignat}{5}
    (\nu \nabla u, \nabla v)_{\Omega_1 \cup \Omega_2} + \langle \lambda, [v]\rangle_\Gamma &= \langle j, v \rangle_{\Omega_1 \cup \Omega_2} \qquad && \forall v \in V, \label{harutyunyan:eq:weak1}\\
    \langle [u], \mu\rangle_\Gamma &= 0 \qquad && \forall \mu \in M. \label{harutyunyan:eq:weak2}
\end{alignat}
Here $(a,b)_{\Omega_1 \cup \Omega_2}=\int_{\Omega_1} a \cdot b \, dx + \int_{\Omega_2} a \cdot b \, dx$ is the usual scalar product of functions $a,b \in L^2(\Omega_1 \cup \Omega_2)$, while $\langle a,b\rangle_{\Omega_1 \cup \Omega_2}$ and $\langle a,\b\rangle_{\Gamma}$ are the duality products on $V \times V'$ and $M \times M'$, respectively, with $V'$, $M'$ denoting the dual spaces of $V$ and $M$. 
Furthermore, $H^1(\Omega_1\cup\Omega_2)$ denotes the space of piecewise smooth functions $v$ with restrictions $v_\ell=v|_{\Omega_\ell} \in H^1(\Omega_\ell)$ for $\ell=1,2$ and $[v]=v_1-v_2$ denotes the jump of such functions across the interface $\Gamma$. 
Under the above assumptions, we have
\begin{lemma} \label{harutyunyan:lem:weak}
For any $j_s \in L^2(\Omega)$ and $m \in L^2(\Omega)^2$, the variational problem  \eqref{harutyunyan:eq:weak1}--\eqref{harutyunyan:eq:weak2} with $j = j_s + \div m^\perp$ has a unique solution $(u,\lambda) \in V \times M$ and there holds  
\begin{align*}
    \|u\|_{H^1(\Omega_1 \cup \Omega_2)} + \|\lambda\|_{H^{-1/2}(\Gamma)} 
       \le C \big( \|j_s\|_{L^2(\Omega_1 \cup \Omega_2)} + \|m\|_{L^2(\Omega_1 \cup \Omega_2)}).
\end{align*}
Moreover, $u$ is the unique weak solution of \eqref{harutyunyan:eq:sys1}--\eqref{harutyunyan:eq:sys4} and $\lambda=\n \cdot (\nu \nabla u_\ell)$ the associated tangential component of the magnetic field strength $\h$ at the interface.
\end{lemma}
\begin{remark} \label{harutyunyan:rem:infsup}
The result is well-known and a similar assertion can already be found in the work of Babuska \cite{Babuska73}. 
Using Brezzi's theory for saddlepoint problems \cite{Brezzi74}, the essential ingredient turns out to be the inf-sup stability condition
\begin{align} \label{harutyunyan:eq:infsup}
    \sup_{v \in V} \frac{\langle \mu,[v]\rangle_\Gamma}{\|v\|_{H^1(\Omega_1 \cup \Omega_2)}} \ge \beta \|\mu\|_{H^{-1/2}(\Gamma)} 
\end{align}
which has to hold for all $\mu \in M$ with a uniform constant $\beta>0$.
Following \cite{Raviart77}, condition \eqref{harutyunyan:eq:infsup} can be proven as follows: 
Let $z_1 \in H^1(\Omega_1)$ be the weak solution of the mixed boundary value problem
\begin{align} \label{harutyunyan:eq:z1}
    -\Delta z_1 = 0 \quad \text{in } \Omega_1  
\quad \text{with} \quad 
    z_1 = 0 \quad \text{on } \Sigma_1 
\quad \text{and} \quad 
    \partial_n z_1 = \mu \quad \text{on } \Gamma.    
\end{align}
Then by standard arguments for elliptic problems \cite{Babuska73,BenBelgacem99}, one can show that 
\begin{align*}
    \|z_1\|_{H^1(\Omega_1)} \le c_2 \|\mu\|_{H^{-1/2}(\Gamma)} 
\quad \text{and} \quad 
    \langle \mu,z_1\rangle_\Gamma \ge c_1 \|\mu\|_{H^{-1/2}(\Gamma)}^2,
\end{align*}
with positive constants $c_1$, $c_2$ only depending on $\Omega_1$, $\Sigma_1$, and $\Gamma$.\\
Now define $z \in H^1(\Omega_1 \cup \Omega_2)$ by $z=z_1$ on $\Omega_1$ and $z=0$ on $\Omega_2$. Then 
\begin{align*}
    \langle \mu,[z]\rangle_\Gamma = \langle \mu, z_1\rangle_\Gamma\ge c_1 \|\mu\|^2_{H^{-1/2}(\Gamma)} \ge \frac{c_1}{c_2} \|\mu\|_{H^{-1/2}(\Gamma)} \|z_1\|_{H^1(\Omega_1)}.  
\end{align*}
The result now follows with $\beta=\frac{c_1}{c_2}$ by noting that $\|z\|_{H^1(\Omega_1\cup\Omega_2)} = \|z_1\|_{H^1(\Omega_1)}$. \qed
\end{remark}

\noindent
\textbf{Discretization.}
As a next step, we now consider Galerkin approximations of the weak formulation \eqref{harutyunyan:eq:weak1}--\eqref{harutyunyan:eq:weak2}:
Find $u_h \in V_h \subset V$ and $\lambda_N \in M_N \subset M$ such that 
\begin{alignat}{5}
    (\nu \nabla u_h, \nabla v_h)_{\Omega_1 \cup \Omega_2} + \langle \lambda_N, [v_h]\rangle_\Gamma &= \langle j, v_h \rangle_{\Omega_1 \cup \Omega_2} \qquad && \forall v_h \in V_h, \label{harutyunyan:eq:weak1h}\\
    \langle [u_h], \mu_N\rangle_\Gamma &= 0 \qquad && \forall \mu_N \in M_N. \label{harutyunyan:eq:weak2h}
\end{alignat}
The subspaces $V_h$, $M_N$ are always assumed to be finite dimensional in the following.
\begin{lemma} \label{harutyunyan:lem:weakh}
Let the conditions of Lemma~\ref{harutyunyan:lem:weak} be valid and assume that 
\begin{align} \label{harutyunyan:eq:infsuph}
    \sup_{v_h \in V_h} \frac{\langle \mu_N,[v_h]\rangle_\Gamma}{\|v_h\|_{H^1(\Omega_1 \cup \Omega_2)}} \ge \beta' \|\mu_N\|_{H^{-1/2}(\Gamma)} 
\end{align}
for all $\mu_N \in M_N$ with some constant $\beta'>0$. 
Then the variational problem \eqref{harutyunyan:eq:weak1h}--\eqref{harutyunyan:eq:weak2h} has a unique solution $u_h \in V_h$, $\lambda_N \in M_N$. Furthermore
\begin{align} \label{harutyunyan:eq:error}
    \|u-u_h&\|_{H^1(\Omega_1 \cup \Omega_2)} + \|\lambda - \lambda_N\|_{H^{-1/2}(\Gamma)} \\
    &\le C \,  \big( \inf_{v_h \in V_h} \|u-v_h\|_{H^1(\Omega_1 \cup \Omega_2)} + \inf_{\mu_N \in M_N}\|\lambda - \mu_N\|_{H^{-1/2(\Gamma)}} \big) \notag
\end{align}
with $C$ depending only on $\beta'$ in condition \eqref{harutyunyan:eq:infsuph}, the bounds for $\nu$ and the geometry. 
\end{lemma}
\begin{remark} \label{harutyunyan:rem:weakh}
The conditions required for the proof of the corresponding result on the continuous level, except the inf-sup stability condition, are inherited by the Galerkin approximation. The existence of a unique solution can thus again be deduced from Brezzi's saddlepoint theory \cite{Brezzi74}. 
The error estimate \eqref{harutyunyan:eq:error} follows from Galerkin orthogonality and standard arguments; we refer to \cite{Braess99,Brezzi74} for details. 
Hence any choice of approximation spaces $V_h$, $M_N$ that allows to prove the discrete inf-sup stability condition \eqref{harutyunyan:eq:infsuph} will lead to a stable discretization with quasi-optimal error estimates.
\end{remark}

\section{Harmonic stator-rotor coupling} \label{harutyunyan:sec:3}

We now consider a particular class of Galerkin approximations  \eqref{harutyunyan:eq:weak1h}--\eqref{harutyunyan:eq:weak2h} in which $V_h$ is constructed by piecewise polynomials, while the Lagrange multiplier space $M_N$ is defined by trigonometric polynomials. Our analysis in particular also covers the harmonic-coupling of the methods considered in \cite{Bontinck18,DeGersem04}. 

Using polar coordinates, the computational domain $\Omega = \Omega_1 \cup \Omega_2$ can be represented as the image of a rectangle  $\widehat \Omega$ under a mapping $F: \widehat \Omega \to \Omega$; see Figure~\ref{harutyunyan:fig:mapping}.
\begin{figure}[t]
    \vspace*{-1.0em}
    \centering
    \includegraphics{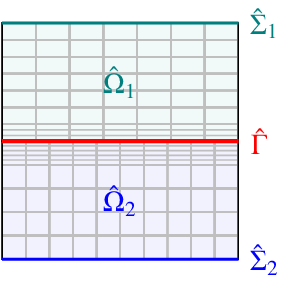}
    \hspace*{6em}
    \includegraphics{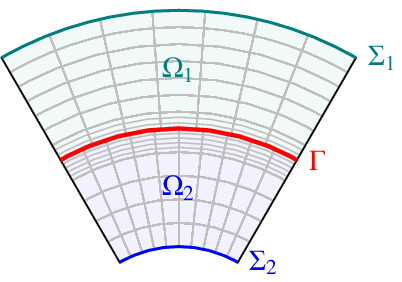}
    \vspace*{-0.5em}
\caption{Sketch of a subset of the rectangular reference domain $\hat \Omega$ and its mesh (left) and the physical domain $\Omega=F(\hat \Omega)$ and mesh obtained after mapping. The boundaries on the left and right are only introduced for the illustration but not present in our application. \label{harutyunyan:fig:mapping}}
    \vspace*{-1.0em}
\end{figure}
Now let $\widehat T_h$ denote a shape-regular partition of $\widehat \Omega_1 \cup \widehat \Omega_2$ into triangles and/or rectangles of size $h$. The meshes of the two sub-domains are assumed to be geometrically conforming, but they may be non-matching across the interface. 
We denote by $P_k(\widehat T_h)$ the space of piecewise polynomials over $\widehat T_h$ of degree $\le k$ and by $\widehat M_N = \operatorname{span}\{\sin(n \pi \xi), \cos(n \pi \xi) : 0 \le n \le N\}$ 
the spaces of trigonometric polynomials of degree $\le N$. 
We then choose the approximation spaces $V_h$, $M_N$ s.t.
\begin{align}
    M_N = F(\widehat M_N) 
\qquad \text{and} \qquad 
    V_h = V_h|_{\Omega_1} \cup V_h|_{\Omega_2} \subset F(P^k(T_h)) \cap V. 
\end{align} 
By the condition $V_h = V_h|_{\Omega_1} \cup V_h|_{\Omega_2}$ we mean that discrete functions, when restricted to one of the sub-domains and extended by zero to the other still belong to the approximation space $V_h$. 
The basic assumption for the discrete inf-sup stability condition \eqref{harutyunyan:eq:infsup} of the corresponding Galerkin approximation \eqref{harutyunyan:eq:weak1h}--\eqref{harutyunyan:eq:weak2h} is the following. 
\begin{theorem} \label{harutyunyan:thm:main}
Assume that there exists a linear operator $\Pi_h : V|_{\Omega_1} \to V_h|_{\Omega_1}$ such that 
\begin{align}
    \|\Pi_h v_1\|_{H^1(\Omega_1)} 
               &\le c_3 \|v_1\|_{H^1(\Omega_1)}, \label{harutyunyan:eq:cond1}\\
    \Pi_h v_1 &= \pi_h v_1 \quad \text{on } \Gamma, \label{harutyunyan:eq:cond2} \\
    \|v-\pi_h v\|_{H^{-1/2}(\Gamma)} &\le c_4 \tfrac{h}{k} \|v\|_{H^{1/2}(\Gamma)}, \label{harutyunyan:eq:cond3}
\end{align}
where $\pi_h : L^2(\Gamma) \to V_h|_{\Omega_1 \cap \Gamma}$ denotes the $L^2$-projection on $\Gamma$. 
Then there exists a constant $0 < \eps <1$, depending only on $c_3$, $c_4$ in \eqref{harutyunyan:eq:cond1}--\eqref{harutyunyan:eq:cond3}, such that the discrete inf-sup condition \eqref{harutyunyan:eq:infsuph} holds with $\beta'=\beta'(\eps)$ whenever $N$ and $h$ are chosen such that
\begin{equation}\label{harutyunyan:eq:Nhk}
    N h/k \le 1-\eps.
\end{equation} 
\end{theorem}
\begin{proof}
As an immediate consequence of the continuous inf-sup condition \eqref{harutyunyan:eq:infsup}, we can find for $\mu = \pi_h \lambda_N$ a function $z \in V$ with $z=0$ on $\Omega_2$, such that 
\begin{align*}
    \langle \pi_h \lambda_N, [z] \rangle_\Gamma \ge \beta \|z\|_{H^1(\Omega_1 \cup \Omega_2)} \|\pi_h \lambda_N\|_{H^{-1/2}(\Gamma)}. 
\end{align*}
We then define $z_h = \Pi_h z_1$ on $\Omega_1$ and $z_h=0$ on $\Omega_2$, and observe that 
\begin{align*}
    \langle \pi_h \lambda_N, [z_h] \rangle_\Gamma
    &= \langle \pi_h \lambda_N, [\Pi_h z] \rangle_\Gamma
     = \langle \pi_h \lambda_N, \pi_h [z] \rangle_\Gamma
     = \langle \pi_h \lambda_N, [z] \rangle_\Gamma, 
\end{align*}
where we used property \eqref{harutyunyan:eq:cond2} and the orthogonality of the $L^2$-projection $\pi_h$. 
Together with the previous estimate and employing condition \eqref{harutyunyan:eq:cond1}, we thus obtain 
\begin{align*}
    \langle \pi_h \lambda_N, [z_h] \rangle_\Gamma 
     &\ge  \beta \|z\|_{H^1(\Omega_1 \cup \Omega_2)} \|\pi_h \lambda_N\|_{H^{-1/2}(\Gamma)}
     \ge \frac{\beta}{c_3} \|z_h \|_{H^1(\Omega_1 \cup \Omega_2)} \|\pi_h \lambda_N\|_{H^{-1/2}(\Gamma)}.
\end{align*}
Using the triangle inequality, we can further estimate 
\begin{align*}
    \|\pi_h \lambda_N\|_{H^{-1/2}(\Gamma)} \ge \|\lambda_N\|_{H^{-1/2}(\Gamma)} - \|\lambda_N - \pi_h \lambda_N\|_{H^{-1/2}(\Gamma)},    
\end{align*}
and the last term can be bounded with the approximation error estimate \eqref{harutyunyan:eq:cond3} by 
\begin{align*}
    \|\lambda_N - \pi_h \lambda_N\|_{H^{-1/2}(\Gamma)} 
    \le c_4 \tfrac{h}{k} \|\lambda_N\|_{H^{1/2}(\Gamma)} 
    \le C' \tfrac{h}{k} N \|\lambda_N\|_{H^{-1/2}(\Gamma)}.  
\end{align*}
In the second estimate, we here used an inverse inequality for the finite dimensional Lagrange multiplier space $M_N$. In summary, we thus obtain  
\begin{align*}
    \langle \pi_h \lambda_N, [z_h] \rangle_\Gamma 
    \ge \beta (1-C' N h/k) \|z_h\|_{H^1(\Omega_1 \cup \Omega_2)} \|\lambda_N\|_{H^{-1/2}(\Gamma)},  
\end{align*}
from which the assertion of the theorem follows immediately.
\end{proof}
\begin{remark}
The conditions of the theorem hold for a variety of discretization methods, e.g. FEM or IGA.
The projection operator $\Pi_h$ can here be constructed following the ideas of \cite{Clement75,Scott90} or \cite{Buffa16} and the 
approximation property \eqref{harutyunyan:eq:cond3} for $\pi_h$ is well-known; details will be given in a forthcoming publication.
The resulting harmonic-coupling mortar methods are thus stable, if the number of degrees of freedom $n \sim k/h$ located at the interface  exceeds the number of coupling modes $N$ to some extent, cf. \eqref{harutyunyan:eq:Nhk}. 
Our main arguments may be applied to other problems and discretization strategies.
\end{remark}

\vspace*{-2em}

\section{Numerical results} \label{harutyunyan:sec:4}

We now illustrate the theoretical results of Theorem~\ref{harutyunyan:thm:main} by some numerical tests using an IGA discretization \cite{Hughes_2005aa} as implemented in GeoPDEs \cite{Falco_2011aa}.  The geometry used in our computations is depicted in Figure~\ref{harutyunyan:fig:geometry}. 
Following the arguments given in Remark~\ref{harutyunyan:rem:infsup} and underlying the proof of Theorem~\ref{harutyunyan:thm:main}, we have 
\begin{align} \label{harutyunyan:eq:infsup2}
\sup_{v \in V} \frac{\langle \mu,v\rangle_\Gamma}{\|v\|_{H^1(\Omega_1 \cup \Omega_2)}} \ge \sup_{z_1 \in V_1} \frac{\langle \mu,z_1\rangle_\Gamma}{\|z_1\|_{H^1(\Omega_1)}} 
\ge \beta \|\mu\|_{H^{-1/2}(\Gamma)},
\end{align}
where $V=\{v \in H^1(\Omega_1 \cup \Omega_2 : v|_{\Sigma_1 \cup \Sigma_2}=0\}$ and $V_1=\{v \in V : v|_{\Omega_2} = 0\} \subset V$.
Due to the Dirichlet boundary conditions on $\Sigma_1$, we can choose $\|v_1\|_{H^1(\Omega_1)}=\|\nabla v_1\|_{L^2(\Omega_1)}$ as the norm on $V_1$. 
One can then show that the second supremum in \eqref{harutyunyan:eq:infsup2} is attained by the solution $z_1$ of the mixed boundary value problem \eqref{harutyunyan:eq:z1}. 
For our model problem, $\Omega_1=\{x \in \RR^2 : R_1 < |x| < R_2\}$ is a simple annulus with radii $R_1 = 0.0447$ and $R_2=0.0675$, and the solution of the above problem can be computed analytically in the form of a Fourier series, and we define $S_1 \mu := v_1|_{\Gamma}$. 
The largest possible constant $\beta$ such that the second estimate of \eqref{harutyunyan:eq:infsup2} remains true for all $\mu \in H^{-1/2}(\Gamma)$ 
can then be characterized by the minimal eigenvalue of 
\begin{align*}
\langle \mu, S_1 \widetilde \mu\rangle_{H^{-1/2}(\Gamma) \times H^{1/2}(\Gamma)} = \beta^2 (\mu,\tilde \mu)_{H^{-1/2}(\Gamma)}^2 \qquad \forall \tilde \mu \in H^{-1/2}(\Gamma).
\end{align*}
For the problem under consideration, the solution can be computed explicitly which gives $\beta=\sqrt{R_1 \ln(R_2/R_1)} \approx 0.13573$. 
The discrete inf-sup constant is evaluated by numerically solving the corresponding discretized eigenvalue problem. 

\bigskip
In the first series of tests, we utilize the lowest order approximation and consider a sequence of uniformly refined meshes. The discrete inf-sup constant is computed as outlined above. The results of these computations are depicted in Table~\ref{harutyunyan:tab:h}. 
\begin{table}[ht!]
\begin{center}
\vskip-1.5em
\setlength{\tabcolsep}{1em}
\renewcommand{\arraystretch}{1.2}
\scriptsize
\caption{Discrete inf-sup constants obtained for $n$ gridpoints at the interface $\Gamma$ and harmonic order $N=c n$ of the Lagrange-multipliers for different refinement levels $\ell$ and scaling parameters $c$. \label{harutyunyan:tab:h}}
\begin{tabular}{c||c|c|c|c}
$c$ $\setminus$ $\ell$ & 1 & 2 & 3 & 4  \\
\hline
\hline
1/4     & 0.135237 & 0.135556 & 0.135676 & 0.135693 \\
\hline 
1/3     & 0.135237 & 0.135556 & 0.135661 & 0.135684 \\
\hline 
3/8     & 0.135237 & 0.135536 & 0.135611 & 0.135684 \\
\hline 
1/2     & 3.526e-08 & 2.532e-08 & 2.401e-08 & 2.401e-08 
\end{tabular}
\end{center}
\vskip-1.5em
\end{table}
The coarsest mesh has $n=144$ vertices at the interface $\Gamma$ and which is doubled in every refinement step; see Figure~\ref{harutyunyan:fig:geometry}.
For $c=1/2$, we have $\dim(M_N) = 2N+1 = n+1 > n$, and the discrete inf-sup stability condition is violated. 
The results of Table~\ref{harutyunyan:tab:h} thus perfectly agree with the theoretical predictions of Theorem~\ref{harutyunyan:thm:main}. 
In a second sequence of tests, we study the dependence of the inf-sup constant on the polynomial degree $k$ of the spline approximation on the mesh with refinement level $2$. The corresponding results are summarized in Table~\ref{harutyunyan:tab:p}. 
\begin{table}[ht!]
\begin{center}
\vskip-1.5em
\setlength{\tabcolsep}{1em}
\renewcommand{\arraystretch}{1.2}
\scriptsize
\caption{Discrete inf-sup constants for $n$ spline degrees of freedom on the interface $\Gamma$ and harmonic order $N=c n$ of the Lagrange-multipliers for polynomial degree $k$ and scaling parameter $c$.\label{harutyunyan:tab:p}}
\begin{tabular}{c||c|c|c|c}
$c$ $\setminus$ $k$ & 2 & 3 & 4 & 5 \\
\hline
\hline
1/4     & 0.135721 & 0.135723 & 0.135723 & 0.135723 \\
\hline 
1/3     & 0.135721 & 0.135722 & 0.135723 & 0.135723 \\
\hline 
3/8     & 0.135720 & 0.135723 & 0.135723 & 0.135723 \\
\hline 
1/2     & 3.652e-08 & 0 & 8.082e-08 & 1.825e-08 \\ 
\end{tabular}
\end{center}
\vskip-1.5em
\end{table}
For the choice $c=1/2$, the number of Lagrange-multipliers $2N+1 =n+1> n$ again exceeds the number of the spline degrees at the interface $\Gamma$ and the discrete inf-sup stability fails. The computational results are again in perfect agreement with the theoretical  predictions.

\vspace*{-1em}
\begin{acknowledgement}
This work is supported by the `Excellence Initiative' of the German Federal and State Governments and by the Graduate School of Computational Engineering at Technische Universit\"at Darmstadt and the grants TRR~154 project C04 and TRR~146 project C03.
\end{acknowledgement}

\end{document}